\documentclass[11pt]{article}
\usepackage[a4paper, margin=2cm]{geometry}
\usepackage{enumerate}
\usepackage{amsmath}
\usepackage{amssymb,latexsym}

\usepackage{amsthm}
\usepackage{color}
\usepackage{graphicx}
\usepackage{amscd}
\usepackage{hyperref}
\usepackage{multicol}
\usepackage{textgreek}
\usepackage{comment}
\usepackage{lineno}

\title{The classification of maximum scattered linear sets of $\PG(1,q^5)$}
\author{G. Longobardi, V. Pepe}

\newcommand{\cV}{{\mathcal V}}

\newcommand{\F}{{\mathbb F}}

\newcommand{\N}{\mathrm{N}}

\newcommand{\xx}{\mathbf{x}}

\newcommand{\Tr}{\mathrm{Tr}}

\newtheorem{theorem}{Theorem}[section]
\newtheorem{lemma}[theorem]{Lemma}
\newtheorem{corollary}[theorem]{Corollary}
\newtheorem{proposition}[theorem]{Proposition}

\DeclareMathOperator{\tr}{Tr}

\DeclareMathOperator{\PG}{{PG}}

\DeclareMathOperator{\PGL}{{PGL}}
\DeclareMathOperator{\PGaL}{P\Gamma L}
\DeclareMathOperator{\GaL}{\Gamma L}

\DeclareMathOperator{\rk}{rk}

\theoremstyle{definition}

\begin{document}
\date{}
\maketitle

\begin{abstract}
\noindent We classify the maximum scattered $\mathbb F_q$-linear sets of $\PG(1,q^5)$, proving that every such set is of pseudoregulus type or of Lunardon–Polverino type. Building on the reduction obtained by Lia, Longobardi and Zanella in \cite{llz25}, we show that the two remaining candidate families contain no scattered linear sets. Our approach reduces these candidates to two normal forms and establishes their non-scatteredness through the existence of rational points on associated algebraic varieties. 
\end{abstract}


\medskip


\section{Introduction}

Let $q$ be a prime power, let $n\geq2$, and let $V$ be an
$r$-dimensional vector space over $\F_{q^n}$, with $r\geq2$.

An $\F_q$-subspace $U$ of $V$ of dimension $m$ defines the
\emph{$\F_q$-linear set of rank $m$}
\[
 L_U=\{\langle\mathbf u\rangle_{\F_{q^n}}:
              \mathbf u\in U\setminus\{\mathbf0\}\}
       \subseteq\PG(V,\F_{q^n})=\PG(r-1,q^n).
\]
For a point $P=\langle\mathbf v\rangle_{\F_{q^n}}$, its
\emph{weight with respect to $U$} is
\[
 w_U(P)=\dim_{\F_q}\bigl(U\cap\langle\mathbf v\rangle_{\F_{q^n}}\bigr).
\]
Every point of $L_U$ has positive weight, and
\[
 |L_U|\leq\frac{q^m-1}{q-1}.
\]
Equality holds precisely when all the points of $L_U$ have weight one;
in this case $L_U$ is called \emph{scattered}.
Blokhuis and Lavrauw proved that the rank of a scattered linear set
in $\PG(r-1,q^n)$ is at most $rn/2$~\cite{BL2000}.
In particular, the largest possible rank on the projective line
$\PG(1,q^n)$ is $n$. A scattered linear set of this rank is called a
\emph{maximum scattered linear set}.

Two linear sets in $\PG(r-1,q^n)$ are \emph{equivalent} if a
collineation in $\PGaL(r,q^n)$ maps one onto the other, and
\emph{projectively equivalent} if such a map can be chosen in
$\PGL(r,q^n)$. These notions concern the point sets: equivalence of
$L_U$ and $L_W$ need not imply that $U$ and $W$ lie in the same
$\GaL(r,q^n)$-orbit. This distinction is relevant both to
classification questions and to the connections with rank-metric
codes; see~\cite[Sections~1.3 and~3]{Longobardi2026}.

Linear sets connect finite projective geometry with blocking sets,
translation structures, semifields and rank-metric codes. On the
projective line, linearized polynomials provide a particularly useful
description. After a projectivity, any $\F_q$-linear set of rank $n$
in $\PG(1,q^n)$ can be written as
\[
 L_f=\{\langle(x,f(x))\rangle_{\F_{q^n}}:x\in\F_{q^n}^*\},
 \qquad
 f(X)=\sum_{i=0}^{n-1}a_iX^{q^i}\in\F_{q^n}[X],
\]
where $\F_{q^n}^*=\F_{q^n}\setminus\{0\}$.
The polynomial $f$ is called \emph{scattered} if $L_f$ is scattered,
or, equivalently, if
\[
 \frac{f(x)}x=\frac{f(y)}y
 \quad\Longrightarrow\quad
 \frac xy\in\F_q
 \qquad(x,y\in\F_{q^n}^*).
\]
Sheekey's construction associates with a scattered polynomial the code
\[
 \mathcal C_f=\{aX+bf(X):a,b\in\F_{q^n}\},
\]
viewed as a space of $\F_q$-linear endomorphisms of $\F_{q^n}$.
After choosing an $\F_q$-basis, this is an $\F_q$-linear maximum rank
distance (MRD) code of $n\times n$ matrices, of dimension $2n$ and
minimum rank distance $n-1$~\cite{sheekey16}.
For an account of these connections and the known constructions, we
refer to the survey~\cite{Longobardi2026}.

Two constructions are central to the classification problem considered
here. The linear sets of \emph{pseudoregulus type} are the sets
projectively equivalent to
\begin{equation}\label{intro:pseudoregulus}
 L^{\mathrm{pr}}_s=
 \{\langle(x,x^{q^s})\rangle_{\F_{q^n}}:x\in\F_{q^n}^*\},
 \qquad \gcd(s,n)=1;
\end{equation}
see~\cite{BL2000,scatt}.

For $n\geq4$, the linear sets of \emph{Lunardon--Polverino type},
or \emph{LP type}, are the sets projectively equivalent to
\begin{equation}\label{intro:LP}
 L^{\mathrm{LP}}_{s,\delta}=
 \{\langle(x,x^{q^s}+\delta x^{q^{n-s}})\rangle_{\F_{q^n}}:
                  x\in\F_{q^n}^*\},
\end{equation}
where $1\leq s<n$, $\gcd(s,n)=1$ with $\mathrm{N}_{q^n/q}(\delta) \not \in \{0,1\}$.
The original construction is due to Lunardon and
Polverino~\cite{LP}, and the generalized construction occurs in
Sheekey's work~\cite{sheekey16}.

The classification of maximum scattered linear sets of $\PG(1,q^n)$ is complete for $n\leq4$.
For $n=2$, every maximum scattered $\F_q$-linear set is a Baer
subline.
For $n=3$, every maximum scattered linear set is of pseudoregulus
type; this follows from the classification of rank-three linear sets
on a projective line by Lavrauw and Van de Voorde~\cite{LvV2010}.
For $n=4$, Csajb\'ok and Zanella proved that every maximum scattered
linear set is of pseudoregulus type or of LP type~\cite{CZ2018}.

When $q=2$,
every maximum scattered $\F_2$-linear set on $\PG(1,2^n)$ is of
pseudoregulus type~\cite[Section~1.2]{llz25}.

The next case is the classification of maximum scattered linear sets of $\PG(1,q^5)$. Up to equivalence, the only examples known in~\cite{llz25, Longobardi2026, MZ2022} consist of
the pseudoregulus point set and the two LP subfamilies
$L^{\mathrm{LP}}_{1,\delta}$ and $L^{\mathrm{LP}}_{2,\delta}$ .
The two LP subfamilies are inequivalent under
$\PGaL(2,q^5)$~\cite[Section~1]{MZ2022}.

For the remainder of the article, we recall the maps

\[
\Tr_{q^5/q}: x \in \mathbb{F}_{q^5} \mapsto x+x^q+\cdots + x^{q^{4}} \in \mathbb{F}_q
\]

\noindent and

\[
\N_{q^5/q}: x \in \mathbb{F}_{q^5} \mapsto x^{1+q+\cdots + q^{4}} \in \mathbb{F}_q
\]

that are called  the \textit{trace} and \textit{norm} maps of $\mathbb{F}_{q^5}$ over $\mathbb{F}_q$, respectively.
In polynomial terms, the result stated by Lia, Longobardi and Zanella for the classification of maximum scattered linear sets of $\PG(1,q^5)$ is the following.

\begin{theorem}\cite[Corollary~7.2]{llz25}\label{thm:llz25}
Any maximum scattered linear set $L$ in $\PG(1,q^5)$ is, up to
equivalence in $\PGaL(2,q^5)$, one of the following:
\begin{enumerate}[(C1)]

\item an MSLS of pseudoregulus type $L^{\mathrm{pr}}_s$;

\item an MSLS of LP type $L^{\mathrm{LP}}_{s,\delta}$;

\item a linear set of the form
\[
\left\{
\left\langle
\bigl(
\eta(x^q-x)+\Tr_{q^5/q}(\rho x),
x^q-x^{q^4}
\bigr)
\right\rangle_{\F_{q^5}}
:
x\in\F_{q^5}^*
\right\},
\]
where $\eta \in \F_{q^5}^*$,
$\Tr_{q^5/q}(\eta)=0 \neq 
\Tr_{q^5/q}(\rho)$;

\item a linear set of the form
\[
\left\{
\left\langle
\bigl(
x,
k(x^q+x^{q^3})+x^{q^2}+x^{q^4}
\bigr)
\right\rangle_{\F_{q^5}}
:
x\in\F_{q^5}^*
\right\},
\]
where $\mathrm{N}_{q^5/q}(k)=1\neq k$.

\end{enumerate}
\end{theorem}

The classes (C3) and (C4) describe the remaining candidates, and the
reduction does not assert that they contain scattered members.
An exhaustive computation in~\cite[Remark~6.5 and Theorem~7.1]{llz25}
excludes such members for every prime power $q\leq25$.
Thus the results of~\cite{llz25} establish the classification in
that range and leave the two candidate classes to be excluded for
arbitrary $q>25$.

The aim of the present paper is to complete this classification by
addressing these remaining classes. We develop a common algebraic
and geometric approach to (C3) and (C4).
Our first reduction shows that every member of these classes is
projectively equivalent to a linear set
$L^\varepsilon_\gamma=L_{W^\varepsilon_\gamma}$, where
$\varepsilon\in\{0,1\}$, $\gamma\in\F_{q^5}^*$, and
\begin{align}
 W^0_\gamma&=
 \{(x,x^q+\gamma h):x\in\F_{q^5},\ 
                 \Tr_{q^5/q}(x)=0,\ h\in\F_q\},
                 \label{intro:W0}\\
 W^1_\gamma&=
 \{(x,x^q+\gamma\Tr_{q^5/q}(x)):x\in\F_{q^5}\}.
                 \label{intro:W1}
\end{align}
These auxiliary families are larger than the original candidate
classes and include known scattered examples.
For instance, when $\gamma\in\F_q^*$, $L^0_\gamma$ is of LP
type in characteristic seven, and $L^1_\gamma$ is of LP type if
$7\gamma^2+5\gamma+1=0$, see Proposition \ref{LP}.

The next step translates non-scatteredness into a problem about
solutions of trace equations. A sufficient condition for
$L^\varepsilon_\gamma$ to have a point of weight greater than one
is the existence of $\theta\in\F_{q^5}\setminus\F_q$ satisfying
\begin{equation}\label{intro:trace-system}
 \begin{cases}
 \Tr_{q^5/q}\!\left(
                \dfrac{\gamma}{\theta^q-\theta}\right)=0,\\[6pt]
 \Tr_{q^5/q}\!\left(
                \dfrac{\gamma\theta}{\theta^q-\theta}\right)
                =\varepsilon.
 \end{cases}
\end{equation}
These equations lead us to search for points of an associated algebraic surface in $\PG(4,q^5)$ such that they lie  in the canonical subgeometry $\Sigma$ and are distinct from $P$.
\[
 \Sigma=
 \{\langle(\theta,\theta^q,\theta^{q^2},\theta^{q^3},
                       \theta^{q^4})\rangle_{\F_{q^5}}:
                       \theta\in\F_{q^5}^*\}
 \cong\PG(4,q).
\]
different from $P=\langle(1,1,1,1,1)\rangle_{\F_{q^5}}$.

The paper is organized as follows. In Section~2, we study the two  classes $(C3)$ and $(C4)$ arising from Theorem~\ref{thm:llz25}. We show that, up to projective equivalence,every linear set belonging to either of the two remaining families is projectively equivalent to one of the linear sets $L_\gamma^\varepsilon$ with $\varepsilon\in\{0,1\}$ and $\gamma\in\mathbb F_{q^5}^*$. We also determine  some choices of the parameter $\gamma \in \F_{q^5}$ for which these linear sets are of LP type.

In Section~3, we move the scatteredness issue for $L^\varepsilon_\gamma$ into an algebraic one. More precisely, we derive a sufficient condition for the existence of a point of weight two in terms of the solvability over $\mathbb F_{q^5}$ of a system of trace equations; this leads to the criterion stated in Theorem~\ref{Lepsilon}.

Section~4 is devoted to the geometric study of this system. We associate with $L^\varepsilon_\gamma$ a projective variety $\mathcal V_\varepsilon$ in $\PG(4,q^5)$ and investigate its intersection with the canonical subgeometry $\Sigma\cong\PG(4,q)$. Combining these results with the classification for $q\leq25$ established in~\cite{llz25}, we conclude with Theorem~\ref{th:main-classification}, proving that every maximum scattered $\mathbb F_q$-linear set of $\PG(1,q^5)$ is, up to projective equivalence, either of pseudoregulus type or of LP type.

\

\section{Equivalent forms for the linear sets in the classes $(C3)$ and $(C4)$}

In this section, we show that, up to projective equivalence, the study of
the linear sets of $\PG(1,q^5)$ belonging to the classes $(C3)$ and
$(C4)$ can be reduced to two forms. These forms will be used in the next section to prove that the original linear sets are not scattered.

\subsection{The linear sets of the class \textit{(C3)}}\label{subsec:(C3)}

Let $L_U$ be the linear set of $\PG(1,q^5)$ as in \textit{(C3)} where 

$$U=\left \{\left (\eta(x^q-x)+\Tr_{q^5/q}(\rho x), x^q-x^{q^4} \right ) \colon  x\in\F_{q^5} \right \}$$
with $\eta,\rho \in \F_{q^5}^*$ such that $\Tr_{q^5/q}(\eta)=0 \neq \Tr_{q^5/q}(\rho)$.
Consider the element of $\PGL(2,q^5)$ induced by the linear map 
$$(X,Y) \longmapsto  \left (-\frac{1}{\eta}X+Y,\frac{1}{\eta}X \right ).$$
The image of $L_U$ under this projectivity is  the linear set $L_W$, where 
$$
W=\left \{ \left (x-x^{q^4}-\frac{1}{\eta}\Tr_{q^5/q}(\rho x),x^q-x+\frac{1}{\eta}\Tr_{q^5/q}(\rho x) \right ) \colon x \in \F_{q^5} \right \}.
$$
Let $f: x \in \F_{q^5} \longrightarrow x-x^{q^4}-\frac{1}{\eta}\Tr_{q^5/q}(\rho x) \in \F_{q^5}$, then 

$$
W=\left \{ \left (f(x),f(x)^q+\alpha\Tr_{q^5/q}(\rho x) \right ) \colon x \in \F_{q^5} \right \}
$$
where $\alpha=\frac{1}{\eta} + \frac{1}{\eta^q} \neq 0$. Indeed, if $\alpha=0$, then $\eta^q=-\eta$, and hence $0=\Tr_{q^5/q}(\eta)=\eta$,
a contradiction.\\
\textbf{Case 1.} Assume that $f(x)$ is a bijection of $\F_{q^5}$.  Therefore,  
$$
W=\left \{(y,y^q+\alpha \Tr_{q^5/q}(\rho f^{-1}(y))) \colon y \in \F_{q^5}\right \}.
$$
Since  $\Tr_{q^5/q}(\rho f^{-1}(y))=\Tr_{q^5/q}(\beta y)$ for some $\beta \in \F_{q^5}^*$, we get 
\begin{equation}\label{eq:W1}
W=\left \{(y,y^q+\alpha \Tr_{q^5/q}(\beta y))\colon y \in \F_{q^5} \right \}.
\end{equation}

\medskip

\noindent \textbf{Case 2.} Now, suppose that $f(x)$ is not a bijection and let $K_1$ and $K_2$ be the kernels of the linear maps $\Tr_{q^5/q}(\rho x)$ and $f(x)$, respectively. Then  $f(x)=x-x^{q^4}$ for every $x \in K_1$ and, since $\Tr_{q^5/q}(\rho) \neq 0$, we get $K_1\cap K_2=\{0\}$. Moreover, since $\dim_{\F_q} K_1=4$ and $K_1 \cap K_2 = \{0\}$, we have that $\dim_{\F_q} K_{2}=1$. Hence, $\F_{q^5}=K_1\oplus K_2$ as $\F_{q}$-vector space and we get that
$$
W=\left \{ \left (f(x),f(x)^q+\alpha \Tr_{q^5/q}(\rho x) \right ) \colon x \in \F_{q^5}\right \}= \left \{ \left (f(x_1),f(x_1)^q+\alpha \Tr_{q^5/q}(\rho x_2) \right ) \colon x_i \in K_i \text{ and } i=1,2 \right \}.
$$
Since $K_1\cap K_2=\{0\}$, we have that $f(x)$ is injective on $K_1$ and $\Tr_{q^5/q}(\rho x)$ is injective on $K_2$. Therefore,

\begin{equation}\label{eq:W2}
W=\left \{\left (x,x^q+\alpha h \right ) \colon  x\in \F_{q^5}\text{ with } \Tr_{q^5/q}(x)=0 \text{ and } h \in \F_q \right \}.
\end{equation}
with $\alpha \in \F_{q^5}^*$.
\subsection{The linear sets of the class \textit{(C4)}}

Let $L_U$ be the linear set of $\PG(1,q^5)$ as in \textit{(C4)} where 
 $$U= \left \{ \left (x,k(x^{q}+x^{q^3})+x^{q^2}+x^{q^4}\right ) \colon  x\in\F_{q^5}\right \}$$ 
 with $\N_{q^5/q}(k)=1 \neq k$. Consider the element of PGL$(2,q^5)$ induced by the linear map
$$
(X,Y)\mapsto \frac{1}{k-1}(X+Y,-kX-Y).
$$
The image of $L_U$ under this projectivity is equivalent to the linear set $L_W$, where 

$$W=\left \{\frac{1}{k-1}\left (x+k(x^{q}+x^{q^3})+x^{q^2}+x^{q^4},-k x-k(x^{q}+x^{q^3})-x^{q^2}-x^{q^4}\right ) \colon x \in \F_{q^5}\right \}.$$
Let $f: x \in \F_{q^5} \longrightarrow \frac{1}{k-1}(x+k(x^{q}+x^{q^3})+x^{q^2}+x^{q^4}) \in \F_{q^5}$. It is straightforward to check that 
\begin{equation}\label{eq:W(C4)}
W=\left \{(f(x),f(x)^q+\alpha\Tr_{q^5/q}(x)) \colon x \in \F_{q^5} \right \}
\end{equation}
where $\alpha=-1-\frac{1}{k-1}-\frac{1}{k^q-1}=\frac{1-k^{q+1}}{(k-1)^{q+1}}$. Now, if $\alpha=0$, then $k^{q+1}=1$. Since $\mathrm{N}_{q^5/q}(k)=1$, by Hilbert's Theorem 90, there exists some $w \in \F_{q^5}^*$ such that $k=w^{q-1}$. Then we obtain
$1=k^{q+1}=(w^{q-1})^{q+1}=w^{q^2-1}$ and this shows that $w \in \F_{q}^*$. It follows that $k=w^{q-1}=1$, contradicting the assumption that $k \ne 1$. Therefore, $1-k^{q+1} \ne 0$, and consequently $\alpha \ne 0$.\\
\textbf{Case 1} As in Subsection \ref{subsec:(C3)},  if $f(x)$ is a bijection, 
we get that 
\begin{equation}\label{W3}
W=\left \{(y,y^q+\alpha \Tr(\beta y)) \colon y \in \F_{q^5}\right \}.
\end{equation}
for some $\beta \in \F_{q^5}^*$ .

\noindent \textbf{Case 2} On the other hand, if $f(x)$ is not a bijection, let $K_1$ and $K_2$ be the kernels of the linear maps $\Tr_{q^5/q}(x)$ and $f(x)$, respectively. Then  $f(x)=x^q+x^{q^3}$ for every $x \in K_1$ and it is easy to see that  $K_1\cap K_2=\{0\}$. Since $\dim_{\F_q} K_1=4$ and $K_1 \cap K_2 = \{0\}$, we have that $\dim_{\F_q} K_{2}=1$ and $\F_{q^5}=K_1\oplus K_2$ as $\F_{q}$-vector space. Then, by \eqref{eq:W(C4)},
$$
W=\left \{ (f(x),f(x)^q+\alpha \Tr_{q^5/q}(x))\colon x \in \F_{q^5}\right \}=\left \{(f(x_1),f(x_1)^q+\alpha\Tr_{q^5/q}(x_2)) \colon x_i \in K_i, i=1,2\right \}.
$$
Moreover, since $K_1\cap K_2=\{0\}$, we have that $f(x)$ is injective on $K_1$ and $\Tr_{q^5/q}(x)$ is injective on $K_2$.
Therefore,  

$$W=\left \{(x,x^q+\alpha h) \colon  x\in \F_{q^5} \text{ with } \mathrm{Tr}_{q^5/q}(x)=0 \text{ and } h \in \F_q \right \}$$ 
with $\alpha \in \F_{q^5}^*$.

\vspace{1cm}

\noindent We conclude this section by summarizing the results obtained in the
two previous subsections. More precisely, the linear sets arising from
the families $(C3)$ and $(C4)$ can be reduced, up to projective
equivalence, to two types of linear sets.

Note that, in the cases \eqref{eq:W1} and \eqref{W3}, since $\beta\neq 0$,
putting $z=\beta y$ and applying the projectivity induced by the 
map
\[
(X,Y)\mapsto (\beta X,\beta^q Y),
\]
we get that the corresponding linear set is projectively equivalent to
$L_{W'}$, where
\[
W'=\{(z,z^q+\gamma\Tr_{q^5/q}(z)):z\in\mathbb F_{q^5}\}
\]
and $\gamma=\alpha\beta^q\in\mathbb F_{q^5}^*$.

\begin{proposition}\label{C3-C4}
Let $L_U$ be a linear set of $\PG(1,q^5)$ belonging to one of the families
$(C3)$ or $(C4)$. Then $L_U$ is projectively equivalent to one of the linear sets
$L^\varepsilon_\gamma:=L_{W^\varepsilon_\gamma}$, with $\varepsilon\in\{0,1\}$ and a suitable $\gamma\in\mathbb F_{q^5}^*$, where
\[
W^0_\gamma=\{(x,x^q+\gamma h):x\in\mathbb F_{q^5},
\Tr_{q^5/q}(x)=0,\ h\in\mathbb F_q\},
\]
and
\[
W^1_\gamma=\{(x,x^q+\gamma\operatorname{Tr}_{q^5/q}(x)):x\in\mathbb F_{q^5}\}.
\]
\end{proposition}

We stress that the families $(C3)$ and $(C4)$ are properly contained in the family of linear sets $\{L^\varepsilon_\gamma\colon \varepsilon\in\{0,1\}, \gamma\in\mathbb F_{q^5}^*\}$, as we show in the following proposition.

\begin{proposition}\label{LP}
Let $\gamma\in\F_q^*$. Then $L^\varepsilon_\gamma$ is of LP type
whenever one of the following conditions holds:
\begin{enumerate}
    \item[(a)] $\varepsilon=0$ and $\operatorname{char}(\F_q)=7$;
    \item[(b)] $\varepsilon=1$ and $7\gamma^2+5\gamma+1=0$.
\end{enumerate}
\end{proposition}
\begin{proof}
Let $\varepsilon=0$ and $\mathrm{char}\, (\F_q)=7$, then 

\[
W^0_\gamma=W^0_1=\{(x,x^q+h):x\in\mathbb F_{q^5},
\Tr_{q^5/q}(x)=0,\ h\in\mathbb F_q.\}
\]

Let $z:=3x+x^q+h$. We have $\Tr_{q^5/q}(h)=5h=0$ if and only if $h=0$. Also, since $\mathrm{N}_{q^5/q}(-3)=2 \neq 1$, $3x+x^q=0$ if and only if $x=0$. Therefore, the linear map

$$ (X,Y)\longmapsto(3X+Y,\ 4X+4Y) $$

maps $W^0_1$ to 

$$ W'':=\{(z,z^{q^2}+3z^{q^3}):z\in\mathbb F_{q^5}\} ,$$

which induces a linear set of LP type as $\mathrm{N}_{q^5/q}(3)=3^5=5\neq 1$.\\

Let $\varepsilon=1$ and $\gamma \in \F_q^*$ be such that
$7\gamma^2+5\gamma+1=0$. It straightforward to see that $1+ 2 \gamma \neq 0$, so we may consider
$$ \delta=-\frac{\gamma}{1+2\gamma}. $$

A direct computation, using $7\gamma^2+5\gamma+1=0$, gives
\[
\delta^2-\delta+1=0,
\qquad
\delta+\gamma(1+2\delta)=0.
\]

Now let $f: x \in \F_{q^5}  \longrightarrow \delta x+x^q+\gamma\Tr_{q^5/q}(x) \in \F_{q^5}.$
The determinant of the Dickson matrix associated with $f$ is
\[
(\delta^4-\delta^3+\delta^2-\delta+1)
(\delta+1+5\gamma).
\]
Using $\delta^2-\delta+1=0$,
$\delta=-\gamma/(1+2\gamma)$ and $7\gamma^2+5 \gamma +1 =0$, this becomes
\[
\delta^3(\delta-1)
\frac{\gamma(3\gamma +1)}{1+2\gamma}.
\]

Again, since $7\gamma^2+5\gamma+1=0$, it is easy to see that $3\gamma+1\neq0$. Therefore $f$ is an invertible $\F_q$-linear map. Moreover, the linear map

$$ (X,Y)\longmapsto(\delta X+Y,\ -\delta X-\delta Y) $$

maps $W^1_\gamma$ into 

\[W'=\{(z,z^{q^2}+\delta z^{q^3}) \colon z \in \F_{q^5}\}.\]

Also, since $\delta^2-\delta+1=0$, we have  $\delta^3+1=0$ and hence $\delta^6=1$. Then,  $\N_{q^5/q}(\delta)=\delta^5=\delta^{-1} \neq 1$, getting that $L_{W'}$ is a linear set of LP type.

\end{proof}

\section{Non-scatteredness of the linear sets $L^\varepsilon_\gamma$: an algebraic condition}

By Proposition 3.1, in order to prove that the linear sets arising from the families $(C3)$ and $(C4)$ are
not scattered, it is enough to consider the two $\F_q$-subspaces $W^\varepsilon_{\gamma}$, with $\varepsilon \in \{0,1\}$ and $\gamma \in \F_{q^5}^*$, described there.\\
Let $L^{\varepsilon}_{\gamma}$ be the $\F_q$-linear set of $\PG(1,q^5)$ as in Proposition \ref{C3-C4} and consider the $\F_q$-vector space 
$$\overline{W}=\{(x,x^q) \colon x \in \mathbb{F}_{q^5} \text{ with }  \tr_{q^5/q}(x)=0\}$$ with $\dim_{\F_q} \overline{W}=4$. Clearly, $L_{\overline{W}}$ is a scattered linear set of PG$(1,q^5)$ contained in $L^\varepsilon_{\gamma}$ and in a linear set of pseudoregulus type.
Moreover, the vector space $\overline{W}$ is a hyperplane of $W^\varepsilon_\gamma$ and, by Grassmann's Formula, it is straightforward to see that $L^\varepsilon_{\gamma}$ has points of weight at most 2. \\
Next, we provide a sufficient condition ensuring that
an $\mathbb F_q$-linear set of type $L^\varepsilon_{\gamma}$ for $\varepsilon \in\{0,1\}$,
 has a point of weight $2$.
\bigskip

Let $L^0_\gamma$ be the $\F_q$-linear set of $\PG(1,q^5)$ whose underlying $\F_q$-vector space is
\[
W^0_\gamma=\{(x,x^q+\gamma h): \Tr_{q^5/q}(x)=0  \text{ and } h \in \F_{q}\},
\]
for some $\gamma \in \F_{q^5}^*$. A sufficient condition for $L^0_{\gamma}$
to have a point of weight $2$ is the existence of $u,x \in \F_{q^5}^*$,
with $\Tr_{q^5/q}(u)=\Tr_{q^5/q}(x)=0$
such that the vectors
\[
(x,x^q+\gamma )
\quad\text{and}\quad
(u,u^q)
\]
define the same projective point. 
Hence, as before, $L^0_{\gamma}$ has a point of weight $2$ if there exists
$u\in\F_{q^5}^*$, with $\Tr_{q^5/q}(u)=0$, such that the system

\begin{equation} \label{eq:sL0}
\begin{cases}
xu^q-ux^q=\gamma u\\
\Tr_{q^5/q}(x)=0

\end{cases}
\end{equation}
must have at least one solution. 
\medskip

Now, let $L^1_\gamma$ be the $\F_q$-linear set of $\PG(1,q^5)$ whose underlying $\F_q$-vector space is
\[
W^1_\gamma=\{(x,x^q+\gamma \Tr_{q^5/q}(x)) : x \in \F_{q^5}\},
\]
for some $\gamma \in \F_{q^5}^*$. A sufficient condition for $L^1_{\gamma}$
to have a point of weight $2$ is the existence of $u,x \in \F_{q^5}^*$,
with $\Tr_{q^5/q}(u)=0\neq \Tr_{q^5/q}(x)$,
such that the vectors
\[
(x,x^q+\gamma \Tr_{q^5/q}(x))
\quad\text{and}\quad
(u,u^q)
\]
define the same projective point. Since $\Tr_{q^5/q}(x)\neq0$, putting $y=\frac{x}{\Tr_{q^5/q}(x)}$, then

$$(x,x^q+\gamma \Tr_{q^5/q}(x))=\Tr_{q^5/q}(x)\biggl (\frac{x}{\Tr_{q^5/q}(x)},\frac{x^q}{\Tr_{q^5/q}(x)}+ \gamma \biggr )=\Tr_{q^5/q}(x)(y,y^q+\gamma)$$
with $\Tr_{q^5/q}(y)=1$. Hence, $L^1_{\gamma}$ has a point of weight $2$ if there exists
$u\in\F_{q^5}^*$, with $\Tr_{q^5/q}(u)=0$, such that the system
\begin{equation} \label{eq:sL1}
\begin{cases}
yu^q=(y^q+\gamma)u,\\
\Tr_{q^5/q}(y)=1
\end{cases}
\end{equation}
has a solution $y\in\F_{q^5}$.

\bigskip

Note that Systems~\eqref{eq:sL0} and~\eqref{eq:sL1} can be treated
simultaneously. Indeed, for $\varepsilon\in\{0,1\}$, let $z=x$, if $\varepsilon=0$, or $z=y$, if $\varepsilon=1$. Then both systems are equivalent to
\[
\begin{cases}
u^qz-uz^q=\gamma u,\\
\Tr_{q^5/q}(z)=\varepsilon.
\end{cases}
\]
Taking the successive $q$-powers of the first equation, this system
can be written as
\begin{equation}\label{eq:ls}
\begin{cases}
u^qz-uz^q=\gamma u,\\
u^{q^2}z^q-u^qz^{q^2}=\gamma^qu^q,\\
u^{q^3}z^{q^2}-u^{q^2}z^{q^3}=\gamma^{q^2}u^{q^2},\\
u^{q^4}z^{q^3}-u^{q^3}z^{q^4}=\gamma^{q^3}u^{q^3},\\
-\,u^{q^4}z+uz^{q^4}=\gamma^{q^4}u^{q^4},\\
z+z^q+z^{q^2}+z^{q^3}+z^{q^4}=\varepsilon.
\end{cases}
\end{equation}
This may be regarded as a linear system over $\F_{q^5}$ in the unknowns $(z,z^q,z^{q^2},z^{q^3},z^{q^4})$.
We first show that the coefficient matrix 
\begin{equation}\label{Au}
A_u=
\begin{pmatrix}
u^q & -u & 0 & 0 & 0\\
0 & u^{q^2} & -u^q & 0 & 0\\
0 & 0 & u^{q^3} & -u^{q^2} & 0\\
0 & 0 & 0 & u^{q^4} & -u^{q^3}\\
-u^{q^4} & 0 & 0 & 0 & u\\
1 & 1 & 1 & 1 & 1
\end{pmatrix}
\end{equation}
of the above linear system has rank
$4$. Indeed, the upper-left $4\times 4$ minor is non-zero, and hence the
coefficient matrix has rank at least $4$.
Consider the two $5\times 5$ minors obtained by bordering this upper-left
$4\times 4$ minor. Then, we get
\begin{equation}
\det
\begin{pmatrix}
u^q & -u & 0 & 0  & 0 \\
0 & u^{q^2} & -u^q & 0  & 0\\
0 & 0 & u^{q^3} & -u^{q^2} & 0 \\
0 & 0 & 0 & u^{q^4} & -u^{q^3} \\
-u^{q^4} & 0 & 0 & 0 & u
\end{pmatrix}=\N_{q^5/q}(u)-\N_{q^5/q}(u)=0
\end{equation}
and, since $\Tr_{q^5/q}(u)=0$, we have that
\begin{equation}
\det \begin{pmatrix}
u^q & -u & 0 & 0 & 0\\
0 & u^{q^2} & -u^q & 0 & 0\\
0 & 0 & u^{q^3} & -u^{q^2} & 0\\
0 & 0 & 0 & u^{q^4} & -u^{q^3}\\
1 & 1 & 1 & 1 & 1
\end{pmatrix}=u^{q+q^2+q^3} \Tr_{q^5/q}(u)=0.
\end{equation}

Now, consider the augmented matrix 
\begin{equation*}
C^\varepsilon_{u,\gamma}=
\begin{pmatrix}
                   u^q & -u & 0 & 0 & 0 & \gamma u \\
                   0 & u^{q^2} & -u^q & 0 & 0& \gamma^q u^q \\
                   0&0&u^{q^3}& -u^{q^2}& 0&\gamma^{q^2}u^{q^2}\\
                   0 & 0 &0 & u^{q^4} &-u^{q^3}&\gamma^{q^3}u^{q^3}\\
                   -u^{q^4} & 0 & 0 & 0 & u&\gamma^{q^4}u^{q^4}\\
                  1&1&1& 1&1& \varepsilon

\end{pmatrix}.
\end{equation*}
associated with the linear system in \eqref{eq:ls}. Then, $\det(C^\varepsilon_{u,\gamma})=-\Tr_{q^5/q}(u)\N_{q^5/q}(u)\Tr_{q^5/q}\left (\frac{\gamma}{u^q} \right )$ and this is again equal to 0. Finally, consider the remaining $5\times 5$ minors of $C^\varepsilon_{u,\gamma}$ obtained by bordering the upper-left $4 \times 4$ minor. These are obtained by adding either the fifth or the sixth row, together with the sixth column of $C^\varepsilon_{u,\gamma}$. Then, we get
\[\det \begin{pmatrix}
      u^q & -u & 0 & 0 & \gamma u \\
                   0 & u^{q^2} & -u^q & 0 & \gamma^q u^q \\
                   0&0&u^{q^3}& -u^{q^2}&\gamma^{q^2}u^{q^2}\\
                   0 & 0 &0 & u^{q^4} &\gamma^{q^3}u^{q^3}\\
                   -u^{q^4} & 0 & 0 & 0&\gamma^{q^4}u^{q^4}\\
\end{pmatrix}=
u^{q^4}\N_{q^5/q}(u)
\Tr_{q^5/q}\left(\frac{\gamma}{u^q}\right),
\]

and
\begin{equation}
\begin{aligned}
\det 
 \begin{pmatrix}
      u^q & -u & 0 & 0 & \gamma u \\
                   0 & u^{q^2} & -u^q & 0 & \gamma^q u^q \\
                   0&0&u^{q^3}& -u^{q^2}&\gamma^{q^2}u^{q^2}\\
                   0 & 0 &0 & u^{q^4} &\gamma^{q^3}u^{q^3}\\
                   1 & 1 & 1 & 1 & \varepsilon \\
\end{pmatrix}=&
u^{q+q^2+q^3+q^4}
\Big(
\varepsilon
-u(c+c^q+c^{q^2}+c^{q^3})
-u^q(c^q+c^{q^2}+c^{q^3})\\
&-u^{q^2}(c^{q^2}+c^{q^3})
-u^{q^3}c^{q^3}
\Big).
\end{aligned}
\end{equation}
where $c=\frac{\gamma}{u^q}$.
Thus,  the matrix $C^\varepsilon_{u,\gamma}$ has rank $4$ and hence the linear system in \eqref{eq:ls} has a solution, if and
only if
\[
\operatorname{Tr}_{q^5/q}\left(\frac{\gamma}{u^q}\right)=0
\]
and
\begin{equation}\label{eq:weird}
\varepsilon
=
u(c+c^q+c^{q^2}+c^{q^3})
+u^q(c^q+c^{q^2}+c^{q^3})
+u^{q^2}(c^{q^2}+c^{q^3})
+u^{q^3}c^{q^3}.
\end{equation}
Note that, since $\Tr_{q^5/q}(u)=0$, then there exists $\theta \in \F_{q^5}$
 such that $u=\theta-\theta^{q^4}$ and substituting this expression in \eqref{eq:weird}, we have 
\begin{equation}
    \Tr_{q^5/q}( \theta c)=\varepsilon+\theta^{q^4}\Tr_{q^5/q}(c).
\end{equation}

Hence, there exists a point of weight 2 in $L^{\varepsilon}_{\gamma}$ if  there exists a solution $\theta \in \F_{q^5} \setminus \F_q$  of the following system

\begin{equation}\label{eq:variety1}
\begin{cases}
\Tr_{q^5/q}\left (\frac{\gamma}{\theta^q - \theta}\right )=0\\
\Tr_{q^5/q}\left (\theta \frac{\gamma}{\theta^q - \theta}\right )=\varepsilon.
\end{cases}
\end{equation}
Finally, since $\mathrm{N}_{q^5/q}(\theta^q - \theta) \in \F_q^*$ and  the trace is $\F_q$-linear, we have 
$$\frac{\gamma}{\theta^q-\theta}=\frac{\gamma (\theta^q-\theta)^{q+q^2+q^3+q^4}}{\mathrm{N}_{q^5/q}(\theta^q-\theta)},$$ 

we may state the following.

\begin{theorem}\label{Lepsilon}
Let $\varepsilon \in \{0,1\}$ and $\gamma \in \F_{q^5}^*$. The $\F_q$-linear set $L^\varepsilon_\gamma$ of $\PG(1,q^5)$ has a point  of weight two if there exists $\theta \in \F_{q^5} \setminus \F_q$ such that 
\begin{equation}\label{tr-n-system}
\begin{cases}
\Tr_{q^5/q}(\gamma(\theta^q-\theta)^{q+q^2+q^3+q^4})=0\\
\Tr_{q^5/q}(\gamma\theta(\theta^q-\theta)^{q+q^2+q^3+q^4})=\varepsilon\N_{q^5/q}(\theta^q-\theta)
\end{cases}
\end{equation}                    
\end{theorem}

\section{The variety associated with $L_\gamma^\varepsilon$}
 By Theorems~\ref{Lepsilon}, in order to prove that the
$\F_q$-linear set $L^\varepsilon_\gamma$, $\varepsilon \in\{0,1\}$, is not scattered,
it is enough to show the existence of a solution
$\theta\in \F_{q^5}\setminus \F_q$ of the corresponding system of
equations over $\F_{q^5}$.
In this section, we reformulate the condition \eqref{tr-n-system} in geometric terms,
we investigate the existence of points that belong to certain algebraic varieties. Let
\[
\Sigma=
\left\{
\langle
(\theta,\theta^q,\theta^{q^2},\theta^{q^3},\theta^{q^4})
\rangle_{\F_{q^5}}
:\theta\in\F_{q^5}
\right\}
\cong \PG(4,q)
\]
be the subgeometry of $\PG(4,q^5)$ of order $q$ that is fixed point wise by the collineation of
$\PG(4,q^5)$ induced by the semilinear map
\begin{equation}\label{sigma}
\sigma:
(x_0,x_1,x_2,x_3,x_4)
\longmapsto
(x_4^q,x_0^q,x_1^q,x_2^q,x_3^q).
\end{equation}

In what follows, by a slight abuse of notation, we denote by
$\sigma$ both the semilinear map and the induced collineation.


let $\gamma \in \F_{q^5}^*$ and define the following  homogeneous polynomial of degree 4 in $\F_{q^5}[X_0,X_1,X_2,X_3,X_4]$:
\begin{equation}\label{F1}
F_1(X_0,X_1,X_2,X_3,X_4)=
\sum_{i=0}^4 \left (
\gamma^{q^i}\prod^4_{j=1}(X_{i+j+1}-X_{i+j})
\right),
\end{equation}
and the following two homogeneous polynomials of degree 5:
\begin{equation}\label{F21}
F_2^{(\varepsilon)}(X_0,X_1,X_2,X_3,X_4)=
\sum_{i=0}^4\left( 
\gamma^{q^i} X_i\prod^4_{j=1}(X_{i+j+1}-X_{i+j})
\right)
-\varepsilon \prod_{i=0}^4(X_{i+1}-X_i),
\end{equation}
for $\varepsilon \in \{0,1\}$, belonging to $\F_{q^5}[X_0,X_1,X_2,X_3,X_4]$. 

 Let 
\[
\mathcal{X}_\varepsilon:=\cV_\varepsilon\cap\Sigma .
\]
where $\cV_\varepsilon:=V_{\F_{q^5}}(F_1,F_2^{(\varepsilon)})$ denotes the projective variety of $\PG(4,q^5)$ defined by $F_1$ and
$F_2^{(\varepsilon)}$.   We observe that $F_1(\xx^{\sigma})=F_1(\mathbf{x})^q$ and $F_2^{(\varepsilon)}(\mathbf{x}^{\sigma})=F_2^{(\varepsilon)}(\mathbf{x})^q$ for any $\xx \in \F_{q^5}^5$. Hence, we get that $\cV_\varepsilon^{\sigma}=\cV_\varepsilon$. Therefore, the varieties $\cV_\varepsilon$ and $\mathcal{X}_\varepsilon:=\cV_\varepsilon\cap\Sigma$ have the same degree and dimension, that are the ones of $\overline{\mathcal{V}}_{\varepsilon}=V(F_1,F_2^{(\varepsilon)})$ defined on the algebraic closure of $\F_{q^5}$, say $\F$.

We also observe that the point $P= (1,1,1,1,1) $ of $\Sigma$ belongs to $\mathcal{X}_\varepsilon$, $\varepsilon=0,1$, and it
corresponds to $\theta\in\F_q$. Hence, by Theorems~\ref{Lepsilon}, the existence of a point of $\mathcal{X}_\varepsilon$ different from $P$ ensures that $L_{\gamma}^\varepsilon$ is not scattered. \\

 By the invertible linear map  $X_{i+1}-X_i\mapsto Y_i$ for $i=0,1,2,3$ and $X_4 \mapsto T$, we get that

\[
\mathbb{F}[X_0,X_1,X_2,X_3,X_4]\simeq \mathbb{F}[Y_0,Y_1,Y_2,Y_3,T].
\]

Also, we note that $X_0-X_4 \mapsto -(Y_0+Y_1+Y_2+Y_3)$
For the remainder of this section, all the identities involving the variables $Y_0,\ldots,Y_4$ are understood in the ring \[ \mathcal{A}:= \frac{\mathbb {F}[Y_0,Y_1,Y_2,Y_3,Y_4,T]} {(Y_0+Y_1+Y_2+Y_3+Y_4)}\simeq  \mathbb{F}[Y_0,Y_1,Y_2,Y_3,T]. \]

\begin{lemma}\label{irrid}
The polynomial
$$F_1(X_0,X_1,X_2,X_3,X_4)=
\sum_{i=0}^4 \left (
\gamma^{q^i}\prod^4_{j=1}(X_{i+j+1}-X_{i+j})
\right) \in \F_{q^5}[X_0,X_1,X_2,X_3,X_4],$$
where $\gamma \in \F_{q^5}^*$, is irreducible over $\F$.
\end{lemma}

\begin{proof}
 Then, the polynomial $F_1(X_0,X_1,X_2,X_3,X_4)$ is mapped in $\mathcal{A}$ to
\begin{equation}
    F_1(Y_0,Y_1,Y_2,Y_3,Y_4)=AY_0^2+BY_0+C
\end{equation}
where
\[
A:=A(Y_1,Y_2,Y_3)=-\gamma^qY_2Y_3-\gamma^{q^2}Y_1Y_3-\gamma^{q^3}Y_1Y_2,
\]
\[B:=B(Y_1,Y_2,Y_3)=-(Y_1+Y_2+Y_3)(\gamma^qY_2Y_3+\gamma^{q^2}Y_1Y_3+\gamma^{q^3}Y_1Y_2)+(\gamma^{q^4}-\gamma)Y_1Y_2Y_3\]
and
\[C:=C(Y_1,Y_2,Y_3)=-\gamma Y_1Y_2Y_3(Y_1+Y_2+Y_3).\]

It is easy to see that $A$ is absolutely irreducible. Indeed, if $A$ were factorizable, one of the two factors would have to be independent of $Y_3$. Such a factor would then have to divide both

$$ \gamma^qY_2+\gamma^{q^2}Y_1, \qquad \gamma^{q^3}Y_1Y_2. $$
and since these are coprime, we have a contradiction.
Suppose that $F_1$ decomposes into two factors linear in $Y_0$. Then, we  have

\[
F_1=(AY_0+D)(Y_0+E).
\]
where 
$$B=(\gamma^{q^4}-\gamma)Y_1Y_2Y_3+(Y_1+Y_2+Y_3)A=AE+D,$$
for some homogeneous polynomials $D,E \in \F[Y_1,Y_2,Y_3,Y_4]$ with $\deg D =3$ and $\deg E =1$. Since $D\mid C$ and $\deg D=3$, we must have $D=dY_1Y_2Y_3$ or $D=dY_iY_j(Y_1+Y_2+Y_3)$, with $d  \in \F$, $i,j \in \{1,2,3\}$ and distinct. If $(\gamma^{q^4}-\gamma)Y_1Y_2Y_3-D \neq 0$, then $Y_iY_j \mid ((\gamma^{q^4}-\gamma)Y_1Y_2Y_3-D)=A(E-Y_1-Y_2-Y_3)$,   for some $i,j \in \{1,2,3\}, i \neq j$. Now, since  $\deg (E-Y_1-Y_2-Y_3) = 1$,  it cannot be divided by $Y_iY_j$.  This implies that there exists $r \in \{1,2,3\}$ such that $Y_r \mid A$, a contradiction. Hence, $D=(\gamma^{q^4}-\gamma)Y_1Y_2Y_3$. Then we get $E=Y_1+Y_2+Y_3$, but 
$$(\gamma^{q^4}-\gamma)Y_1Y_2Y_3(Y_1+Y_2+Y_3)=DE=C=-\gamma Y_1Y_2Y_3(Y_1+Y_2+Y_3)$$ implies that $\gamma=0$, a contradiction to the hypothesis. 

Finally, suppose that $F_1= AY_0^2+BY_0+C$ decomposes as follows:

\[
F_1=(Y_0^2+B'Y_0+C')A .
\]

Then, $A\mid \gcd(B,C)$ and that is impossible as $C$ is the product of four linear factors and $A$ is absolutely irreducible.

\end{proof}

\noindent Now, put $ Y_i:=X_{i+1}-X_i$ with $i\in\mathbb Z_5$ and note that $\sum_{i=0}^4Y_i=0$. We have
\begin{equation}\label{FY}
F_1=
\sum_{i=0}^4
\gamma^{q^i}\prod_{\substack{j\in\mathbb Z_5\\ j\neq i}}Y_j,
\quad \quad \text{and} \quad \quad    F_2^{(\varepsilon)}
=
\sum_{i=0}^4
\biggl (\gamma^{q^i}X_i
\prod_{\substack{j\in\mathbb Z_5\\ j\neq i}}Y_j\biggr)
-
 \varepsilon \prod_{i=0}^4Y_i.
 \end{equation} 

We also have the following result. 

\begin{proposition} \label{coprime}
Let $F_1$ and $F_2^{(\varepsilon)}$, $\varepsilon\in\{0,1\}$, be the homogeneous
polynomials defined in \eqref{F1}, \eqref{F21}.
Then $F_1$ and $F_2^{(\varepsilon)}$ are coprime in
$\mathbb F[X_0,X_1,X_2,X_3,X_4]$.
\end{proposition}
\begin{proof}
Since $F_1$ is irreducible over $\mathbb F$, in order
to prove that $F_1$ and $F_2^{(\varepsilon)}$ are coprime it is enough to
show that $F_1\nmid F_2^{(\varepsilon)}$. Then, suppose by contradiction that
$F_1\mid F_2^{(\varepsilon)}$. Since $\deg F_2^{(\varepsilon)}=5$ and $\deg F_1=4$,
there exists a linear form $L^{(\varepsilon)}$ such that  $F_2^{(\varepsilon)}=L^{(\varepsilon)} F_1$ for $\varepsilon \in \{0,1\}$. Let $Y_r=X_{r+1}-X_r$, with indices taken modulo $5$. Then, we get
\[
        F_1|_{Y_r=0}
        =
        \gamma^{q^r}\prod_{j\neq r}Y_j
\]
and
\[
        F_2^{(\varepsilon)}|_{Y_r=0}
        =
        \gamma^{q^r}X_r\prod_{j\neq r}Y_j.
\]
Hence, $L^{(\varepsilon)}\equiv X_r \pmod{Y_r}$ for every $r\in\mathbb Z_5$. Equivalently,
\[
        L^{(\varepsilon)}-X_r\in (X_{r+1}-X_r)
        \qquad\text{for every }r\in\mathbb Z_5.
\]
where $\varepsilon \in \{0,1\}$. This is impossible for a linear form: for $r=0$, the
form $L^{(\varepsilon)}$ is supported on $X_0,X_1$, while for $r=2$ it is
supported on $X_2,X_3$. Hence, such $L^{(\varepsilon)}$ does not exist. Thus
, $F_1\nmid F_2^{(\varepsilon)}$, and since $F_1$ is irreducible, we have the claim.
\end{proof}

By Lemma \ref{irrid} and Proposition \ref{coprime},
 the variety  $\overline{\cV}_\varepsilon$ is a complete intersection in $\PG(4,\overline{\mathbb F})$.
Consequently,
\[
        \dim \overline{\cV}_\varepsilon=2
        \qquad\text{and}\qquad
        \deg \overline{\cV}_\varepsilon=\deg(F_1)\deg(F_2^{(\varepsilon)})=4\cdot 5=20.
\]

Now, we study the reducibility of the variety $\overline{\cV}_\varepsilon$, $\varepsilon \in \{0,1\}$. We observe that the variety $\overline{\cV}_\varepsilon$ is also fixed by collineation induced by $\sigma$ in $\PG(4,\F)$.\\
Consider the planes of $\PG(4,\F)$
\[
\alpha:=V(Y_0,Y_1),\qquad \beta:=V(Y_0,Y_2).
\]
It is easy to see that $\alpha^{\sigma^j}\cap \Sigma= \beta^{\sigma^j}\cap \Sigma=P$ for any $j \in \{0,1,2,3,4\}$.

Consider the polynomial 
\begin{equation}
G=(\gamma+\gamma^{q^4})Y_2Y_3Y_4
-\gamma^{q^2}Y_1Y_3Y_4
-\gamma^{q^3}Y_2^2Y_4
-\gamma^{q^3}Y_1Y_2Y_4
+\gamma^{q^4}Y_0Y_2Y_3.
\end{equation}

and set

$$
K:=G+\varepsilon Y_2Y_3Y_4,
\qquad \varepsilon\in\{0,1\}.
$$

It is straightforward to check the following identities
\begin{equation}\label{eq:F12-K-new}
F_1=-Y_0G+Y_2G^{\sigma}=-Y_0K+Y_2K^{\sigma}
\quad \text{ and }  \quad F_2^{(\varepsilon)}=X_1F_1-Y_0Y_1G-\varepsilon\prod_{i \in \mathbb Z_5}Y_i=X_1F_1-Y_0Y_1K.
\end{equation}

Moreover, since $F_1^{\sigma}=F_1$, we get

\begin{equation}\label{eq:rec}
F_1=-Y_0K+Y_2K^{\sigma}=-Y_iK^{\sigma^{i}}+Y_{i+2}K^{\sigma^{i+1}}
\end{equation}
for any $i \in \mathbb Z_5$.

\begin{theorem}\label{th:decomposition-residual}
Let $\overline{\cV}_{\varepsilon}=V(F_1,F_2^{(\varepsilon)}) \subseteq \PG(4,\mathbb F)$, with $\varepsilon\in\{0,1\}$, and set
\[
\overline{\mathcal{Y}} := V\bigl( K,K^\sigma,K^{\sigma^2},K^{\sigma^3},K^{\sigma^4} \bigr).
\]
Then, $\overline{\cV}_{\varepsilon}$ decomposes as the union
\[
\overline{\cV}_{\varepsilon} = \left( \bigcup_{i=0}^4\alpha^{\sigma^i} \right) \cup \left( \bigcup_{i=0}^4\beta^{\sigma^i} \right) \cup \overline{\mathcal{Y}},
\]
where the planes $\alpha^{\sigma^i}$ appear with multiplicity $2$, the planes $\beta^{\sigma^i}$ with multiplicity $1$. Finally, $\deg \overline{\mathcal{Y}}=5$.
\end{theorem}

\begin{proof}
We first determine the support of
$\overline{\cV}_{\varepsilon}$. Recall that  
$\dim \overline{\cV}_{\varepsilon}=2$  and 
$\deg \overline{\mathcal{V}}_\varepsilon=
20$ and set
\[
\mathcal P
=
\left(
\bigcup_{i=0}^4\alpha^{\sigma^i}
\right)
\cup
\left(
\bigcup_{i=0}^4\beta^{\sigma^i}
\right).
\]
Since
\[
\alpha^{\sigma^i}=V(Y_i,Y_{i+1}),
\qquad
\beta^{\sigma^i}=V(Y_i,Y_{i+2}),
\]
with indices taken modulo $5$, it is easy to see that 
\begin{equation}\label{eq:all-coordinate-planes}
\mathcal P
=
\bigcup_{0\le i<j\le4}V(Y_i,Y_j).
\end{equation}

We shall show that every plane $V(Y_i,Y_j)$ with $i\neq j$ is contained in
$\overline{\cV}_{\varepsilon}$. Indeed, every monomial occurring in
$F_1$ contains at least one of $Y_i,Y_j$, and the same is true for
every monomial occurring in $F_2^{(\varepsilon)}$. Hence $\mathcal P\subseteq\overline{\cV}_{\varepsilon}$ and, by \eqref{eq:rec},
\[
\overline{\mathcal{Y}}= \mathcal{V}(K,K^\sigma,K^{\sigma^2},K^{\sigma^3},K^{\sigma^4})\subseteq\overline{\cV}_{\varepsilon}.
\]

Now, let $Q\in
\overline{\cV}_{\varepsilon}\setminus\mathcal P$.
We claim that $Y_i(Q)\neq0$ for every $i\in\mathbb Z_5$.
Indeed, we have
\[
F_1\big|_{Y_i=0}
=
\gamma^{q^i}
\prod_{j\neq i}Y_jm.
\]
for every $i \in \mathbb{Z}_5$. Thus, if $Y_i(Q)=0$, since $F_1(Q)=0$ and $\gamma\neq0$, there must
exist $j\neq i$ such that $Y_j(Q)=0$. This would imply  $Q\in V(Y_i,Y_j)\subseteq\mathcal P$,
a contradiction. 
Since $Q\in\overline{\cV}_{\varepsilon}$,
from the second identity in
\eqref{eq:F12-K-new} we obtain
\[
Y_0(Q)Y_1(Q)K(Q)=0.
\]
Hence, $K(Q)=0$.
Using the first identity in
\eqref{eq:rec}, we then obtain
\[
Y_2(Q)K^\sigma(Q)=0,
\]
and therefore $K^\sigma(Q)=0.$
Applying successively
\eqref{eq:F12-K-new}, and using again that all the
$Y_i(Q)$ are nonzero, we get $K^{\sigma^i}(Q)
=0$ for every $i \in \mathbb{Z}_5$ and $Q\in\overline{\mathcal{Y}}$. Hence, we get that
$\overline{\cV}_{\varepsilon}
=
\mathcal P\cup\overline{\mathcal{Y}}$.
We now compute the multiplicities the planes $\{\alpha^{\sigma^i},i=0,1,\ldots,\}$ in $\overline{\mathcal{V}}_\varepsilon$.

Consider first $\alpha=V(Y_0,Y_1)$.

By 

\[(F_1,F_2^{(\varepsilon)})=(F_1,X_1F_1-Y_0Y_1K)=(F_1,Y_0Y_1K),\]

we have that $\overline{\mathcal V}_{\varepsilon}=V(F_1,Y_0Y_1K)$. Let $F_2'$ be $Y_0Y_1K$, then we can consider the Jacobian of $(F_1,F_2')$ in $\mathcal{A}$, that is

\[J=\left(\begin{array}{ccccc}
   \frac{\partial F_1}{\partial Y_0}  &  \frac{\partial F_1}{\partial Y_1}  &   \frac{\partial F_1}{\partial Y_2}  &   \frac{\partial F_1}{\partial Y_3}  &   \frac{\partial F_1}{\partial T}    \\
   \frac{\partial F_2'}{\partial Y_0}  &  \frac{\partial F_2'}{\partial Y_1}  &   \frac{\partial F_2'}{\partial Y_2}  &   \frac{\partial F_2'}{\partial Y_3}  &   \frac{\partial F_2'}{\partial T}  
\end{array}\right).\]

In $\mathcal{A}$ we have that $F_2'=F_2'(Y_0,Y_1,Y_2,Y_3)$, hence $\frac{\partial F_2'}{\partial T}$ is identically 0. It is easy to see that 

\[\frac{\partial F_2'}{\partial Y_i}|_{\alpha}=0 \qquad \forall i=0,1,2,3.\]

Therefore, rank$(J|_{\alpha}) \leq 1$ and hence  $\alpha$ has multiplicity $\mu \geq2$ in $\overline{\mathcal{V}}_{\varepsilon}$. The map $\sigma$ is an automorphism for $\overline{\mathcal{V}}_{\varepsilon}$, hence each plane of the orbit $\{\alpha^{\sigma^i} \colon i=0,1,\ldots,4\}$ has the same multiplicity $\mu$. The same is true for the $\sigma$-orbt of $\beta$ and let $\nu$ be the multiplicity of $\beta^{\sigma^i},i=0,1,\ldots,4$. Since $\overline{\mathcal{V}}_{\varepsilon}$ has pure dimension 2, $\dim \overline{\mathcal{Y}}=2$. 

\[20=\deg \overline{\mathcal{V}}_{\varepsilon}= 5 \mu +5\nu+  \deg \overline{\mathcal{Y}} \geq 10 + 5\nu+  \deg \overline{\mathcal{Y}} .\]

Therefore, $\mu=2,\nu=1,\deg \overline{\mathcal{Y}}=5$.








\end{proof}

\begin{corollary}
The variety $\overline{\mathcal{Y}}=V(K,K^{\sigma},K^{\sigma^2},K^{\sigma^3},K^{\sigma^4})$ has dimension 2 and degree 5. Also,  $\overline{\mathcal{Y}}=\overline{\mathcal{Y}}^{\sigma}$, hence it a variety of $\Sigma$, and it is a cone with a vertex $P=(1,1,1,1,1) \in \Sigma$.
\end{corollary}

\begin{proof}

It is straightfoward to see that $\overline{\mathcal{Y}}^\sigma=\overline{\mathcal{Y}}$. Each $K^{\sigma^i}$ depends only on the differences $Y_j=X_{j+1}-X_j$. Let $P=(1,1,1,1,1)$ and consider a point 
$R=(x_0,\ldots,x_4)$  and every $\lambda \in \F$. We have
\[
Y_j(R+\lambda P)=Y_j(R)
\]
for every $j \in \mathbb{Z}_5$. Thus, if $R\in\overline{\mathcal{Y}}$, then
$R+\lambda P\in\overline{\mathcal{Y}}$ for every $\lambda \in \F$.
Equivalently, for every point $R\in\overline{\mathcal{Y}}$ the whole projective
line $\langle P,R\rangle$ is contained in $\overline{\mathcal{Y}}$.
Hence, $\overline{\mathcal{Y}}$ is a cone with vertex $P$.
\end{proof}

\begin{corollary}\label{contain}
None of the planes of $\{\alpha^{\sigma^i},i=0,1,\ldots,4\}$ and $\{\beta^{\sigma^i},i=0,1,\ldots,4\}$ are contained in $\overline{\mathcal{Y}}$.
\end{corollary}
\begin{proof}
 From the proof of
Theorem~\ref{th:decomposition-residual},
\[
K|_{\alpha}
=
Y_2Y_4
\left(
(\gamma+\gamma^{q^4}+\varepsilon)Y_3
-\gamma^{q^3}Y_2
\right)
\not\equiv0
\]
and
\[
K|_{\beta}
=
-\gamma^{q^2}Y_1Y_3Y_4
\not\equiv0.
\]
Hence $\alpha,\beta\not\subseteq\overline{\mathcal{Y}}$, and, by
$\sigma$-invariance, the same holds for all their conjugates.   
\end{proof}

We will need the following result, originally proved in \cite{HH82}  for characteristic 0, and then in \cite{Ballico18} for positive characteristic:

\begin{theorem}\label{ballico}
Let $n\geq 3$ and $\mathbb K$ be a field. The dimension of the $\mathbb K$-vector space of hypersurfaces of degree $d$ vanishing on a set of $e$ generic pairwise disjoint lines  in $\PG (n,\mathbb K)$ is given by:

$$
\max \left \{\binom{d+n}{n}-e(d+1),0 \right \}.
$$
\end{theorem}

Now, we are ready for the following.

\begin{theorem}\label{th:reducible-Y}
Assume that $\overline{\mathcal{Y}}
=V\bigl(K,K^\sigma,K^{\sigma^2},K^{\sigma^3},K^{\sigma^4}\bigr)
\subseteq \PG(4,\mathbb F)$
be reducible.
Then at least one of the following holds:
\begin{enumerate}
    \item[(i)] $\overline{\mathcal{Y}}\cap\Sigma$ contains a point distinct from
    $P=(1,1,1,1,1)$;
    \item[(ii)] $\varepsilon=0$,
    $\operatorname{char}(\mathbb F_q)=7$, and
    $\gamma\in\mathbb F_q^*$;
    \item[ (iii)] $\varepsilon=1$, $\gamma\in\mathbb F_q^*$, and
    \[
        7\gamma^2+5\gamma+1=0.
    \]
\end{enumerate}
In cases $(ii)$ and $(iii)$, the linear set
$L_\gamma^\varepsilon$ is of LP type, as described in
Proposition~\ref{LP}.
\end{theorem}

\begin{proof}
Since $\overline{\mathcal Y}^{\,\sigma}=\overline{\mathcal Y}$,
the map $\sigma$ permutes the geometric irreducible components of
$\overline{\mathcal Y}$. Let $Z$ be one of these components and set $r=\min\{j\geq 1: Z^{\sigma^j}=Z\}$.
Since every geometric irreducible component is defined over a finite
extension of $\mathbb F_{q^5}$, let $n\geq1$ be the smallest integer
such that $Z$ is defined over $\mathbb F_{q^{5n}}$. Then, $r\mid 5n$, getting that 
$n=\frac{r}{\gcd(r,5)}$. There are only two possibilities:
\[
r=n
\qquad\text{or}\qquad
r=5n.
\]

Suppose that no geometric irreducible component of
$\overline{\mathcal{Y}}$ has a $\sigma$-orbit of type $r=5n$.
Then every geometric irreducible component belongs to an orbit
satisfying $r=n$. In this case $\gcd(r,5)=1$. Therefore  $\varphi=\sigma^5$
acts transitively on the same set of components
$Z,Z^\sigma,\ldots,Z^{\sigma^{r-1}}.$
Consequently, the union
\[
W_Z=\bigcup_{i=0}^{r-1} Z^{\sigma^i}
\]
is defined over $\mathbb F_{q^5}$ and is irreducible over
$\mathbb F_{q^5}$, since its geometric irreducible components form a
single orbit under
$\operatorname{Gal}(\mathbb F/\mathbb F_{q^5})$.
Moreover, $W_Z^\sigma=W_Z$. Thus, the $\sigma$-orbits satisfying $r=n$ give precisely
$\sigma$-invariant irreducible components of $\mathcal \overline{Y}$ over
$\mathbb F_{q^5}$.

If one of the
corresponding components $W_Z$ has degree $1$, then $W_Z$ is a plane
$\pi$ satisfying $\pi^\sigma=\pi$. 
By~\cite[Lemma~1]{Lu1999}, $\pi\cap\Sigma\simeq\PG(2,q)$,
and therefore $\overline{\mathcal Y}\cap\Sigma$ contains a point
distinct from $P$.
If $\overline{\mathcal{Y}}$ does not contain irreducible components of degree 1,  then $\overline{\mathcal{Y}}$ decomposes into two irreducible components, one of degree 3 and one of degree 2. Suppose that $W_Z$ is the component of degree 2. The minimum degree for an algebraic variety of dimension 2 of PG$(4,\F)$ is 3 (see e.g. \cite[Corollary 18.12]{Harris}), hence $W_Z$ is contained in a hyperplane $H$ of PG$(4,\F)$ such that $H=H^{\sigma}$. Again by \cite[Lemma 1]{Lu1999}, $\pi=\pi^{\sigma}$, $H \cap \Sigma \simeq$ PG$(3,q)$ and hence $W_Z \cap \Sigma$ is a quadratic cone of a PG$(3,q)$. Therefore, by Warning's second Theorem, $|W_Z \cap \Sigma| \geq q$ and hence the $(i)$ holds. 

\medskip
Now, suppose that there exists a geometric irreducible component $Z$
whose $\sigma$-orbit satisfies $r=5n$. The $r$ distinct components in the $\sigma$-orbit of $Z$ all have the
same degree. Since this orbit is contained in $\overline{\mathcal Y}$,
we have $r\deg Z\leq\deg\overline{\mathcal Y}=5$.
Using $r=5n$, we obtain $n=1, r=5$ and $\deg Z=1$. Therefore $\pi:=Z$ is a plane of $\PG(4,\mathbb F)$ defined over $\mathbb F_{q^5}$ and
\[
\overline{\mathcal Y}
=
\bigcup_{i=0}^4\pi^{\sigma^i},
\qquad
\pi:=Z,
\]
where the five planes $\pi^{\sigma^i}$ are pairwise distinct. The ten unordered pairs of planes split into the two $\sigma$-orbits
represented by $\{\{\pi^{\sigma^i},\pi^{\sigma^{i+1}}\} : i \in \mathbb{Z}_5\}$ and
$\{\{\pi^{\sigma^i},\pi^{\sigma^{i+2}}\} : i \in \mathbb{Z}_5\}$. We distinguish the possible intersection patterns.\\

\noindent\textbf{Five planes meeting pairwise only at $P$.} Suppose that   $\overline{\mathcal{Y}}$ consists of 5 planes pairwise meeting just in $P$. Let $H$ be the hyperplane $X_4=0$. Then $P \notin H$ and hence $\mathcal Y':=\overline{\mathcal{Y}}\cap H$ is the union of 5 lines, say $\ell_i:=\pi^{\sigma^i}\cap H$, such that  $\ell_i \cap \ell_j=\emptyset$ $\forall \, i\neq j$. Let  $K_i:=K^{\sigma^i}|_{X_4=0}$, then it is easy to see that $K_i$ is not identically zero $\forall \, i=0,1,\ldots,4$. Hence $\mathcal Y'=V(K_0,K_1,K_2,K_3,K_4)$ is defined by polynomials of degree 3. Suppose that $\{\ell_i \colon i=0,1,\ldots,4\}$ is a set of lines in general position, then by Theorem \ref{ballico} and $n=d=3,e=5$, the set of cubic vanishing on $\mathcal Y'$ is empty, a contradiction.  Hence the lines of $\mathcal Y'$ are not in general position. Three disjoint lines of PG$(3,\mathbb F)$ are contained in a unique hyperbolic quadric $\mathcal Q$, in particular, in a regulus $\mathcal R$ of $\mathcal{Q}$. Suppose that there are 4 lines of $\mathcal Y'$ contained in the same regulus $\mathcal{R}$. Then  them each line of the opposite regulus $\mathcal{R}'$ contains at least 4 points of $\mathcal Y'$. Hence, since $\deg K_i=3$, the polynomial $K_i$ vanishes on each line of $\mathcal{R}'$, $\forall \, i\in\{0,1,2,3,4\}$. Therefore, $\mathcal{Q}\subset \mathcal Y'$, a contradiction to the fact that $\mathcal{Y}$ is the union of 5 lines. 

\medskip 
\noindent \textbf{Five planes with a common line through $P$.} Firstly, suppose that  $\overline{\mathcal{Y}}$ consists of 5 planes pairwise intersecting in a fixed line $\ell$ through $P$. Then, we would have $\ell=\ell^{\sigma}$ and hence $\ell \cap \Sigma$ would be a subline of order $q$, hence $\overline{\mathcal{Y}}$ would contain at least one point other than $P$ in $\Sigma$.

Now, suppose that $\overline{\mathcal{Y}}$ consists of 5 planes pairwise intersecting in a line that is not fixed by $\sigma$. Hence $\overline{\mathcal{Y}}$ is contained in a hyperplane of PG$(4,\mathbb F)$, say $H$. Then we must have $H=H^{\sigma}$. Let $h$ be the linear form defining $H$, hence $h=h^{\sigma}$. If $h$ does not divide the polynomial $K^{\sigma^i}$, then the restriction $K^{\sigma^i}|_h$ defines a cubic hypersurface containing 5 distinct planes, which are hyperplanes in $H$, and that is impossible. Hence, $h|K^{\sigma^i}$ $\forall \, i \in \mathbb{Z}_5$. Then, we get $H =V(h)\subset \overline{\mathcal{Y}}$, a contradiction to $\dim \overline{\mathcal{Y}}=2$.\\

\noindent \textbf{Some planes of $\overline{\mathcal{Y}}$  intersect in a line through $P$ and others intersect just in $P$.}  The unordered 10 pairs of planes of $\overline{\mathcal{Y}}$ form two orbits of size 5 under the action of $\sigma$, one containing $(\pi,\pi^{\sigma})$ and one containing  $(\pi,\pi^{\sigma^2})$. Suppose that $\pi\cap\pi^{\sigma}$ is a line $\ell$ and $\pi\cap \pi^{\sigma^2}=\{P\}$, Hence, $(\pi\cap\pi^{\sigma})^{\sigma^i}=\ell^{\sigma^i}$ and $(\pi\cap\pi^{\sigma^2})^{\sigma^i}=P$ $\forall \, i \in \mathbb Z_5$. Let $H^{\sigma^i}$ be the hyperplane $\langle \pi^{\sigma^i},\pi^{\sigma^{i+1}} \rangle$, $i \in \mathbb Z_5$ and let $h^{\sigma^i}$ be the linear form defining $H^{\sigma^i}$. We observe that $\pi= H \cap H^{\sigma^4}$. Therefore,

\[\bigcup_{i \in \mathbb Z_5}V(h^{\sigma^i},h^{\sigma^{i+1}})=\overline{\mathcal{Y}}.\]

Let $I_3$ be the $\F$- vector spaces of the polynomials of degree $3$ of PG$(4,\F)$ vanishing on $\overline{\mathcal{Y}}$.  It is easy to see that $\{(h^{1+\sigma^2+\sigma^3})^{\sigma^i},i \in \mathbb Z_5\}$ is a subset of $I_3$. We have that $I_3$ is equivalent to the space of cubic polynomials vanishing on $5$ lines forming a pentagon.Let $\F[X_0,X_1,X_2,X_3]_3$ be the the vector space of cubic polynomials of PG$(3,\F)$. Then $\dim \F[X_0,X_1,X_2,X_3]_3=\binom{6}{3}=20$. Vanishing on a line imposes $\binom{4}{3}=4$ linear conditions on $\F[X_0,X_1,X_2,X_3]_3$.  For $5$ lines, we have $20$ linear conditions. However, because the lines form a closed pentagon, they intersect at $5$ distinct vertices. At each vertex, the condition that the cubic vanishes at that point is shared by the two meeting lines, reducing the number of independent conditions by $5$. Therefore, $\dim I_3=20-20+5=5$. Since $\dim_{\mathbb F}\langle K^{\sigma^i}, i \in \mathbb Z_5 \rangle =5$, we have $\langle K^{\sigma^i}, i \in \mathbb Z_5 \rangle_{\F}=I_3$. The polynomials $K^{\sigma^i}$ do not have terms of the form  $X_j^3$, $j=0,1,\ldots,4$, and since $h^{1+\sigma^2+\sigma^3} \in \langle K^{\sigma^i}, i \in \mathbb Z_5 \rangle $, $h^{1+\sigma^2+\sigma^3}$ does not either.\\
Let $h=\sum a_jX_j$ and let $A:=\{j \in \mathbb Z_5 \colon a_j \neq 0\}$. If $|A|\geq 4$, then $ \exists$   $i \in  A\cap (A+2) \cap (A+3)$ and that would imply that  $h^{1+\sigma^2+\sigma^3}$ contains $x_i^3$, a contradiction. Hence, $|A| \leq 3$. Up to the action of $\sigma$, $h \in \langle X_0,X_1,X_2 \rangle_{\mathbb F}$ or $h \in \langle X_0,X_1,X_3 \rangle_{\mathbb F}$.  Since $P \in H$, if  $h \in \langle X_0,X_1,X_2 \rangle_{\mathbb F}$, then  $h \in \langle Y_1,Y_2\rangle_{\mathbb F}$, while if  $h \in \langle X_0,X_1,X_3 \rangle_{\mathbb F}$, then  $h \in \langle Y_0,Y_1+Y_2 \rangle_{\mathbb F}$.
Suppose first that $h\in\langle Y_1,Y_2\rangle$. By
Corollary~\ref{contain}, neither coefficient can vanish; after rescaling,
we can write $ h=Y_1+aY_2$ with $a\ne0$. Put $b=a^{1+q^4}$ and $c=-a$.
Then, the plane $V(h^{\sigma^4},h)=V(Y_0-bY_2,Y_1-cY_2)$.
On yhis plane,
\[
 Y_0=bY_2,
 \qquad Y_1=cY_2,
 \qquad Y_4=-(b+c+1)Y_2-Y_3.
\]
The identity $F_1|_{V(h^{\sigma^4},h)}=0$ gives
\[
 b+c+1=0,
 \qquad
 \gamma^{q^4}=\gamma^{q^3},
 \qquad
 \gamma c+\gamma^qb+\gamma^{q^2}bc=0.
\]
Hence, $\gamma\in\mathbb F_q^*$ and $b+c+bc=0$. It follows that
\[
 bc=1,
 \qquad
 b=c^2,
 \qquad
 c^2+c+1=0.
\]

Using $F_2^{(\varepsilon)}=X_1F_1-Y_0Y_1K$, we obtain
\[
\begin{aligned}
F_2^{(\varepsilon)}\big|_{V(h^{\sigma^4},h)}
=-bcY_2^3\Bigl(&
(\gamma+\gamma^{q^4}-\gamma^{q^2}c+\varepsilon)Y_3Y_4\\
&-\gamma^{q^3}(c+1)Y_2Y_4
+\gamma^{q^4}bY_2Y_3\Bigr).
\end{aligned}
\]
Since $\gamma\in\mathbb F_q$, $b+c+1=0$, and $Y_4=-Y_3$,
the expression above  becomes $F_2^{(\varepsilon)}\big|_{V(h^{\sigma^4},h)}= \bigl(\gamma c-2\gamma-\varepsilon\bigr)Y_3^2.$
Thus $\gamma(c-2)=\varepsilon.$
If $\varepsilon=0$, then $c=2$, and $c^2+c+1=0$. This implis that $\operatorname{char}(\mathbb F_q)=7$ and $\gamma\in\mathbb F_q^*$.
If $\varepsilon=1$, then $c=2+\gamma^{-1}$ and
\[
 0=c^2+c+1
  =7+\frac5\gamma+\frac1{\gamma^2}.
\]
Therefore, $7\gamma^2+5\gamma+1=0$ with $\gamma \in \F_q^*$.
These are precisely  points $(ii)$ and $(iii)$ in the Theorem statement.

It remains to consider
$h\in\langle Y_0,Y_1+Y_2\rangle$. The case
$h=Y_1+Y_2$ reduces to the preceding computation with $a=1$ and it is
impossible because $b+c+1\ne0$. Thus, $h=Y_0+a(Y_1+Y_2),$ with $a\ne0$.

Put
\[
 t=a^q,
 \qquad b=a(1-t),
 \qquad c=-at.
\]
On the plane $V(h,h^\sigma)$,
\[
 Y_0=-bY_2-cY_3,
 \qquad
 Y_1=-t(Y_2+Y_3),
\]
and consequently $Y_4=(b+t-1)Y_2+(c+t-1)Y_3.$
The coefficients of $Y_2^4$ and $Y_3^4$ in the restriction of $F_1$
are, respectively,
\[
 \gamma^{q^3}tb(t+b-1),
 \qquad
 \gamma^{q^2}tc(t+c-1).
\]
Since $t,c,\gamma\ne0$, the second coefficient gives
$t+c-1=0$. If $b=0$, then $t=1$, and
$t+c-1=c\ne0$, a contradiction. Hence $b\ne0$, and the first
coefficient gives $t+b-1=0$. Therefore $b=c$. From
$b=a(1-t)$ and $c=-at$, we obtain $a=0$, again a contradiction.

 Finally, suppose that $\pi\cap \pi^{\sigma}=P$ and $\pi\cap \pi^{\sigma^2}$ is a line. Then let $H=\langle \pi,\pi^{\sigma^2}\rangle$ and $h$ be the linear form defining $H$. We have that  $\pi= H\cap H^{\sigma^3}$. Then

 \[\bigcup_{i \in \mathbb Z_5}V(h^{\sigma^i},h^{\sigma^{i+2}})=\overline{\mathcal{Y}}.\]
 
 In this case, $\{ (h^{1+\sigma+\sigma^2})^{\sigma^i}, i \in \mathbb Z_5\} \subset I_3=\langle K^{\sigma^i} \colon i=0,1,\ldots,4 \rangle_{\F}$.

 Then, we proceed as before to obtain  $h \in \langle Y_1,Y_2\rangle_{\mathbb F}$ or  $h \in \langle Y_0,Y_1+Y_2 \rangle_{\mathbb F}$ . By Corollary \ref{contain}, we can assume that $h=Y_2+a^{q^2}Y_3$ with $a\neq 0$. Hence,  $h^{\sigma^{3}}=Y_0+aY_1$ and $V(h,h^{\sigma^3})=V(Y_2+a^{q^2}Y_3, Y_0+aY_1) \subset \overline{\mathcal{Y}}\subset \overline{\mathcal{V}}_{\varepsilon}$. 

Put $t=a^{q^2}$ and
\[
 A=t(\gamma^qa-\gamma),
 \qquad
 B=a(\gamma^{q^3}t-\gamma^{q^2}),
 \qquad
 C=\gamma^{q^4}at.
\]
Since $Y_4=(a-1)Y_1+(t-1)Y_3,$
the coefficients of $Y_1^2,Y_3^2,Y_1Y_3$ in the quadratic
factor of $F_1|_{V(h,h^{\sigma^3})}$ are
\[
 B(a-1),
 \qquad
 A(t-1),
 \qquad
 A(a-1)+B(t-1)+C,
\]
respectively. If $a=1$, then $t=1$, and the mixed coefficient is
$C=\gamma^{q^4}\ne0$. If $a\ne1$, then $t\ne1$; the first two
coefficients force $A=B=0$, while the mixed coefficient becomes
$C\ne0$. Both alternatives are impossible.

Finally, suppose that $h=Y_0+a(Y_1+Y_2)$ with $a\ne0$.
Put $t=a^{q^3}$ and $S=Y_1+Y_2$. Since
$\pi=V(h,h^{\sigma^3})$, the equations of $\pi$ give
\[
 Y_0=-aS,
 \qquad
 (1-t)Y_3=tS.
\]
If $t=1$, then $a=1$, $S=0$, and hence
\[
 Y_0=0,
 \qquad
 Y_2=-Y_1,
 \qquad
 Y_4=-Y_3.
\]
It follows that $F_1|_\pi=\gamma Y_1^2Y_3^2\ne0,$
a contradiction. Assume that $t\ne1$. Then
\[
 Y_3=\frac{t}{1-t}S,
 \qquad
 Y_4=\left(a-1-\frac{t}{1-t}\right)S.
\]
After substituting in $F_1$, multiplying by $(1-t)^2$, and removing
the common factor $S^2$, the coefficients of $Y_1^2$ and $Y_2^2$
are
\[
 a\gamma^{q^2}t(at-a+1),
 \qquad
 a\gamma^qt(at-a+1).
\]
Their vanishing forces $at-a+1=0.$
After using this relation, the coefficient of $Y_1Y_2$ reduces to  $-\gamma^{q^4}\frac{a-1}{a}$.
Its vanishing would give $a=1$, whereas
$at-a+1=0$ would then give $t=0$, contradicting
$t=a^{q^3}\ne0$. This last case is impossible as well.

Therefore,  either
$\overline{\mathcal{Y}}\cap\Sigma$ contains a point distinct from $P$, or one of
$(ii)$ and $(iii)$ holds. In the latter cases,
Proposition~\ref{LP} shows that $L_\gamma^\varepsilon$ is of LP type.

\end{proof}

\begin{theorem} \label{th:irreducible-Y}
  If $\overline{\mathcal{Y}}$ is  irreducible, then  $\overline{\mathcal{Y}}\cap\Sigma$ contains at least on point distinct from $P$ for $q \geq 16$.
\end{theorem}
\begin{proof}
Assume that $\overline{\mathcal{Y}}$ is irreducible. After a projective change of coordinates identifying $\Sigma$
with the standard $\mathbb F_q$-subgeometry, choose a hyperplane
$H$ defined over $\mathbb F_q$ and not containing $P$.
Since $\overline{\mathcal{Y}}$ is a cone, its reduced hyperplane section
$C=\overline{\mathcal{Y}}\cap H$ is an irreducible projective
curve defined over $\mathbb F_q$, of degree $d\leq5$.
Let $\widetilde C$ be its normalization and let $g$ be its genus.

We claim that $g\leq2$. If $C$ is a line, then $g=0$.
If $C$ spans a plane $\pi$, at least one of the cubic equations
defining $C$ in $H$ has a nonzero restriction to $\pi$;
otherwise, $\pi$ would be contained in their common zero locus.
Consequently, $d\leq3$, and hence $g\leq1$.
Finally, if $C$ spans $H$, the pullback of the hyperplane series
to $\widetilde C$ gives a divisor $D$ of degree $d\leq5$
with $\ell(D)\geq4$.
Clifford's theorem implies that $D$ is nonspecial otherwise

\[
\ell(D)\leq \frac{d}{2}+1<4.
\]
Riemann--Roch therefore gives
\[
g=d+1-\ell(D)\leq d-3\leq2;
\]
see \cite[Chapter~6]{HKT08}.

By the Hasse--Weil bound \cite[Theorem~9.18]{HKT08},
\[
|\widetilde C(\mathbb F_q)|
\geq q+1-2g\sqrt q
\geq q+1-4\sqrt q>0
\qquad\text{for }q\geq16.
\]
The normalization morphism is defined over $\mathbb F_q$,
so the image of an $\mathbb F_q$-rational point of
$\widetilde C$ is an $\mathbb F_q$-rational point of $C$.
Since $P\notin H$, this yields a point of
$\overline{\mathcal{Y}}\cap\Sigma$ distinct from $P$.
\end{proof}

\begin{theorem}
\label{th:main-classification}
Let $L$ be an $\mathbb F_q$-linear set of rank $5$ in
$\PG(1,q^5)$. Then $L$ is maximum scattered if and only if it is  of pseudoregulus or LP type.

\end{theorem}

\begin{proof}

Let $L$ be a maximum scattered $\mathbb F_q$-linear set of
rank $5$ in $\PG(1,q^5)$. By
Theorem~\ref{thm:llz25}, up to projective equivalence, $L$ belongs
to one of the four families $(C1)$--$(C4)$.

If $L$ belongs to either $(C1)$ or $(C2)$, then it is of either of pseudoregulus or LP type. It remains
to exclude the existence of new maximum scattered linear sets in the
families $(C3)$ and $(C4)$. For $q\leq25$, this already follows from
\cite[Remark~6.5 and Theorem~7.1]{llz25}. We may therefore assume $q>25.$
By Proposition~\ref{C3-C4}, every linear set belonging to $(C3)$ or
$(C4)$ is projectively equivalent to a linear set
$L^\varepsilon_\gamma$ with  $\varepsilon\in\{0,1\}$, $ \gamma\in\mathbb F_{q^5}^*$.
Let $\overline{\mathcal{V}}_\varepsilon= \mathcal{P} \cup \overline{\mathcal{Y}}$ the associated variety with 
\[ \overline{\mathcal{Y}}=V\bigl(
 K,K^\sigma,K^{\sigma^2},K^{\sigma^3},K^{\sigma^4}
 \bigr)
\]
Suppose first that $\overline{\mathcal{Y}}$ is absolutely
irreducible. By Theorem~\ref{th:irreducible-Y}, for any $q \geq 16$, there is a point gives $Q\in(\overline{\mathcal{Y}}\cap\Sigma)\setminus\{P\}$.
By the construction of $\overline{\mathcal{Y}}$ and
Theorem~\ref{Lepsilon}, this point produces a point of weight $2$
in $L^\varepsilon_\gamma$. Hence
$L^\varepsilon_\gamma$ is not scattered, contradicting the
assumption on $L$.
Suppose now that $\overline{\mathcal{Y}}$ is reducible. By
Theorem~\ref{th:reducible-Y}, it follows that either $(\overline{\mathcal{Y}}\cap\Sigma)\setminus\{P\}\neq\varnothing,$
or one of the following exceptional conditions holds:
\[\begin{cases}
 \varepsilon=0,\\
 \operatorname{char}(\mathbb F_q)=7,\\
 \gamma\in\mathbb F_q^*,
\end{cases}
\qquad\text{or}\qquad
\begin{cases}
 \varepsilon=1,\\
 \gamma\in\mathbb F_q^*,\\
 7\gamma^2+5\gamma+1=0.
\end{cases}
\]
In the first case, Theorem~\ref{Lepsilon} again implies that
$L^\varepsilon_\gamma$ has a point of weight $2$, contradicting
its scatteredness.

In each of the two exceptional cases,
Proposition~\ref{LP} shows that $L^\varepsilon_\gamma$ is
 a linear set of LP type.
\end{proof}

\section{Acknowledgment}
This work was supported by the Italian National Group for Algebraic and Geometric Structures and their Applications (GNSAGA–INdAM). The first author gratefully acknowledges the hospitality of Dipartimento di Scienze di Base e Applicate per l'Ingegneria, Sapienza Università di Roma.

 \noindent Giovanni Longobardi,\\
	Dipartimento di Matematica e Applicazioni “R. Caccioppoli”\\
	Università degli Studi di Napoli Federico II,\\
	via Cintia, Monte S. Angelo I-80126 Napoli, Italy. \\
	email: 
	\nolinkurl{giovanni.longobardi@unina.it}

\bigskip

    \noindent Valentina Pepe,\\
Dipartimento di Scienze di Base e Applicate per l'Ingegneria,\\\
Sapienza Università di Roma,\\\
via Antonio Scarpa 16, I-00161 Roma, Italy.\\
email:
\nolinkurl{valentina.pepe@uniroma1.it}

\end{document}